\documentclass[12pt]{amsart}
\usepackage{amsthm,amsmath,amssymb}

\usepackage{mathrsfs}
\usepackage{amsfonts}
\usepackage{bbold}
\usepackage[colorlinks=true,citecolor=black,linkcolor=black,urlcolor=blue]{hyperref}
\usepackage[all]{xy}
\everymath{\displaystyle}

\title{Injective Envelopes and Projective Covers\\ of Quivers}

\author{Will Grilliette}

\newtheorem{thm}{Theorem}[subsection]
\newtheorem{prop}[thm]{Proposition}
\newtheorem{cor}[thm]{Corollary}

\theoremstyle{definition}
\newtheorem{defn}[thm]{Definition}
\newtheorem{defn2}{Definition}[section]

\theoremstyle{plain}
\newtheorem{prop2}[defn2]{Proposition}

\theoremstyle{remark}
\newtheorem{ex}[thm]{Example}
\newtheorem{ex2}[defn2]{Example}

\DeclareMathOperator{\ob}{Ob}
\DeclareMathOperator{\edges}{edges}
\DeclareMathOperator{\Ran}{ran}
\DeclareMathOperator{\card}{card}
\DeclareMathOperator{\indep}{indep}

\newcommand{\Set}{\mathbf{Set}}
\newcommand{\Quiv}{\mathbf{Quiv}}
\newcommand{\cat}[1]{\mathscr{#1}}

\begin{document}

\begin{abstract}
This paper characterizes the injective and projective objects in the category of directed multigraphs, or quivers.  Further, the injective envelope and projective cover of any quiver in this category is constructed.
\end{abstract}

\maketitle
\tableofcontents

\section{Introduction}

In several branches of mathematics, the concepts of injectivity and projectivity have found use due to their respective map lifting properties, characterizing several different examples.  Likewise, injective envelopes and projective covers describe different notions of completions.  Lists of such examples can be found in \cite[\S II.9.3, II.9.17, II.9.28]{joyofcats}.

This paper applies these notions to the category of directed multigraphs and their homomorphisms, characterizing the usual classes of injective and projective objects, as well as the respective envelope and cover.  These notions were previously considered in \cite[Ch.\ VI]{williams-1971}, though with a category of undirected graphs and with no discussion of the envelope or cover.  Specifically, the results in Propositions \ref{mono-inj} and \ref{epi-proj} are analogous to \cite[Theorems 6.3 and 6.1]{williams-1971}, respectively, though the differences in the categories prevents the single loop graph from being projective.  Moreover, the proofs of this paper are direct, appealing to the universal properties of known graphs to shorten the proofs.

The author would like to thank Drs.\ Deborah and Tyler Seacrest for the discussions from which this work arose.  In particular, the author attributes to them the coining of the terms ``loaded'' and ``explosion'' from Definitions \ref{loaded} and \ref{explosion}, respectively.

\section{Important Examples}

Recall the category of directed multigraphs, or quivers.  In this discussion, the following terminology and symbology will be used, where $\Set$ denotes the category of sets with functions.

\begin{defn2}[Quiver, {\cite[Definition 2.1]{graphtransformation}}]
A \emph{quiver} is a quadruple $(V,E,\sigma,\tau)$, where $V,E\in\ob(\Set)$ are sets, and $\sigma,\tau\in\Set(E,V)$ are functions.  Elements of $V$ are \emph{vertices}, and $V$ the \emph{vertex set}.  Elements of $E$ are \emph{edges}, and $E$ the \emph{edge set}.  The function $\sigma$ is the \emph{source map}, and $\tau$ the \emph{target map}.  For $e\in E$, $\sigma(e)$ is the \emph{source} of $e$, and $\tau(e)$ the \emph{target} of $e$.
\end{defn2}

\begin{defn2}[Quiver map, {\cite[Definition 2.4]{graphtransformation}}]
Given quivers $G$ and $H$, a \emph{quiver homomorphism} from $G$ to $H$ is a pair $\left(\phi_{V},\phi_{E}\right)$, where $\phi_{V}\in\Set\left(V_{G},V_{H}\right)$ and $\phi_{E}\in\Set\left(E_{G},E_{H}\right)$ satisfy $\phi_{V}\circ\sigma_{G}=\sigma_{H}\circ\phi_{E}$ and $\phi_{V}\circ\tau_{G}=\tau_{H}\circ\phi_{E}$.  The function $\phi_{V}$ is the \emph{vertex map}, and $\phi_{E}$ the \emph{edge map}.
\end{defn2}

For notation, let $\Quiv$ denote the category of quivers with quiver homomorphisms with the usual component-wise composition.  This category inherits a substantial amount of structure from $\Set$.  In particular, the standard universal constructions in $\Quiv$ mirror their counterparts in $\Set$, each done component-wise with the source and target maps defined appropriately as described in \cite[\S2.15]{borceux1}.  The characterizations are summarized in Table \ref{constructions}.  In particular, these characterizations guarantee that $\Quiv$ is complete and cocomplete as a category.

\begin{table}[b]\caption{Universal Constructions in $\Quiv$ and $\Set$}
\label{constructions}
\begin{center}\begin{tabular}{|c||c|c|}\hline
Construction	&	Characterization in $\Set$	&	Characterization in $\Quiv$\\ \hline\hline
equalizer	&	subset	&	subquiver\\ \hline
coequalizer	&	quotient set by	&	quotient quiver by\\
&	 an equivalence relation	&	a quiver equivalence relation\\ \hline
product	&	Cartesian product	&	Kronecker/tensor product\\ \hline
coproduct	&	disjoint union set	&	disjoint union quiver\\ \hline
\end{tabular}\end{center}
\end{table}

As in \cite[p.\ 106]{seic1985}, there are two natural projection functors $V,E:\Quiv\to\Set$, where one either ignores the edge structure or vertex structure, respectively.  Each functor is both a left and a right adjoint to a canonical construction of a quiver.  These constructions correspond to the left and right Kan extensions along each functor when $\Quiv$ is considered as a presheaf topos, done in \cite[\S 3.7]{borceux1}.

These canonical quivers will be the basis for both of the constructions in this paper, so each and its universal property will be described summarily.  The proof of each characterization is routine and will be omitted.

\begin{defn2}[Reflection quivers, {\cite[p.\ 20]{handbookgt}}]\label{reflections}
Fix a set $S$.
\begin{enumerate}

\item Let $\mathbf{0}_{S}:\emptyset\to S$ be the empty function to $S$.  The \emph{independent set of vertices} or \emph{empty quiver} on $S$ is
\[
I(S):=\left(S,\emptyset,\mathbf{0}_{S},\mathbf{0}_{S}\right),
\]
the quiver with vertex set $S$ and no edges.

\item Given $j=0,1$, let $\iota_{j}:S\to\{0,1\}\times S$ by $\iota_{j}(s):=(j,s)$ be the usual inclusions.  The \emph{independent set of edges} on $S$ is the quiver
\[
M(S):=\left(\{0,1\}\times S,S,\iota_{0},\iota_{1}\right).
\]

\item Given $j=0,1$, let $\pi_{j}:S^{2}\to S$ by $\pi_{1}(s,t):=s$ and $\pi_{2}(s,t):=t$ be the usual projections.  The \emph{(directed) complete graph} or \emph{full quiver} on $S$ is the quiver
\[
K(S):=\left(S,S^{2},\pi_{1},\pi_{2}\right).
\]

\item Let $\mathbb{1}:=\{1\}$ and $\mathbf{1}_{S}:S\to\mathbb{1}$ be the constant function from $S$.  The \emph{(directed) bouquet} on $S$ is
\[
B(S):=\left(\mathbb{1},S,\mathbf{1}_{S},\mathbf{1}_{S}\right),
\]
the quiver with edge set $S$ and one vertex.

\end{enumerate}
\end{defn2}

\begin{ex2} For concreteness, consider the set $S=\{a,b,c\}$.  Then, the special quivers above are drawn below.
\begin{enumerate}

\item $I(S)$:
\[\xymatrix{
a	&	b	&	c\\
}\]

\item $M(S)$:
\[\xymatrix{
(0,a)\ar[d]^{a}	&	(0,b)\ar[d]^{b}	&	(0,c)\ar[d]^{c}\\
(1,a)	&	(1,b)	&	(1,c)\\
}\]

\item $K(S)$:
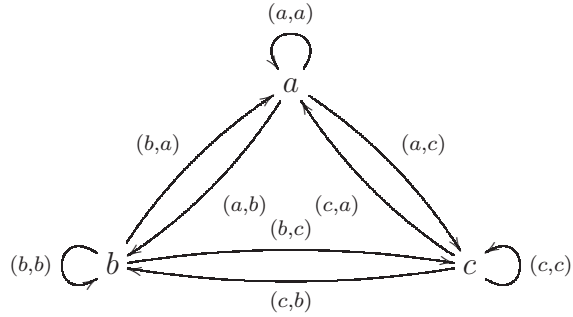
\[\xymatrix{
\\
&	&	a\ar@/^/[ddrr]^{(a,c)}\ar@/^/[ddll]^{(a,b)}\ar@(ur,ul)_{(a,a)}\\
\\
b\ar@/^/[uurr]^{(b,a)}\ar@/^/[rrrr]^{(b,c)}\ar@(ul,dl)_{(b,b)}	&	&	&	&	c\ar@/^/[uull]^{(c,a)}\ar@/^/[llll]^{(c,b)}\ar@(dr,ur)_{(c,c)}\\
}\]

\item $B(S)$:
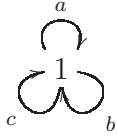
\[\xymatrix{
\\
1\ar@(ul,ur)^{a}\ar@(r,d)^{b}\ar@(d,l)^{c}\\
}\]

\end{enumerate}
\end{ex2}

\begin{prop2}[Universal Characterizations]
Let $G$ be a quiver and $S$ be a set.
\begin{enumerate}

\item Given any function $\phi:S\to V(G)$, there is a unique quiver homomorphism $\hat{\phi}:I(S)\to G$ such that $V\left(\hat{\phi}\right)=\phi$.
\item Given any function $\psi:S\to E(G)$, there is a unique quiver homomorphism $\hat{\psi}:M(S)\to G$ such that $E\left(\hat{\psi}\right)=\psi$.
\item Given any function $\chi:V(G)\to S$, there is a unique quiver homomorphism $\hat{\chi}:G\to K(S)$ such that $V\left(\hat{\chi}\right)=\chi$.
\item Given any function $\xi:E(G)\to S$, there is a unique quiver homomorphism $\hat{\xi}:G\to B(S)$ such that $E\left(\hat{\xi}\right)=\xi$.

\end{enumerate}
\end{prop2}

Since the set $S$ was arbitrary in each of the above constructions, the functors $V$ and $E$ have left adjoint functors $I,M:\Set\to\Quiv$ and right adjoint functors $K,B:\Set\to\Quiv$ defined on objects above.  Pictorially, these can be described below.
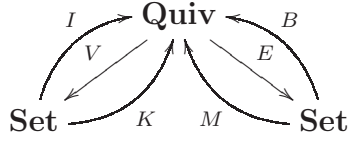
\[\xymatrix{
	&	\Quiv\ar[dl]_{V}\ar[dr]^{E}\\
\Set\ar@/^1.3pc/[ur]^{I}\ar@/_1.3pc/[ur]_{K}	&	&	\Set\ar@/^1.3pc/[ul]^{M}\ar@/_1.3pc/[ul]_{B}\\
}\]
These adjoint characterizations show that the ideas of ``independent set of vertices'', ``independent set of edges'', ``complete graph'', and ``bouquet'' arise naturally from the categorical structure of $\Quiv$.  This reinforces that all these classes of quivers are fundamental to graph theory.

\section{Injectivity}

For a category $\cat{C}$ and class of $\cat{C}$-morphisms $\Phi$, recall from \cite[Definition II.9.22]{joyofcats} that a $\cat{C}$-object $I$ is \emph{$\Phi$-injective} if given any $\phi\in\cat{C}(A,B)$ from $\Phi$ and $\psi\in\cat{C}(A,I)$, there is $\hat{\psi}\in\cat{C}(B,I)$ such that $\hat{\psi}\circ\phi=\psi$.  This is described in the commutative diagram below.
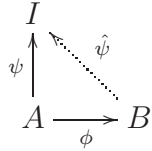
\[\xymatrix{
I\\
A\ar[u]^{\psi}\ar[r]_{\phi}	&	B\ar@{..>}[ul]_{\hat{\psi}}\\
}\]
The $\Phi$-injective envelope is then a ``minimal'' injective embedding.  The notion of minimality here is encoded by a $\Phi$-essential map.  Recall that a map $\phi\in\cat{C}(A,B)$ from $\Phi$ is \emph{$\Phi$-essential} if for all $C\in\ob(\cat{C})$ and $\alpha\in\cat{C}(B,C)$, $\alpha\circ\phi\in\Phi$ implies $\alpha\in\Phi$.

This section considers injectivity of quivers with respect to the class of all monomorphisms.  From \cite[Fact 2.15]{graphtransformation}, a quiver homomorphism $\phi$ is monic if and only if both $V(\phi)$ and $E(\phi)$ are one-to-one.

\subsection{A Motivating Example}\label{motivating}

To begin the discussion of injectivity, first consider the case when the class $\Phi$ of quiver maps is a singleton.  Let $G:=I(\{0,1\})$ be an empty quiver on two vertices and $H:=M(\{e\})$ an independent set of one edge.  These are drawn below.
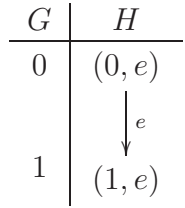
\[
\begin{array}{c|c}
G	&	H\\
\hline
\xymatrix{
0\\
1\\
}
&
\xymatrix{
(0,e)\ar[d]^{e}\\
(1,e)\\
}
\end{array}
\]
Define $\phi_{V}:V(G)\to V(H)$ by $\phi_{V}(t):=(t,e)$.  Then, $\phi:=\left(\phi_{V},\mathbf{0}_{\{e\}}\right)$ is the unique quiver map from $G$ to $H$ extending $\phi_{V}$.

The quivers injective with respect to $\phi$ are characterized by a generalization of a full quiver.

\begin{defn}\label{loaded}
For a quiver $J$ and $v,w\in V(J)$, let
\[
\edges_{J}(v,w):=\sigma_{J}^{-1}(v)\cap\tau_{J}^{-1}(w),
\]
the set of all edges in $J$ with source $v$ and target $w$.\\
A quiver $J$ is \emph{loaded} if for every $v,w\in V(J)$, $\edges_{J}(v,w)\neq\emptyset$.
\end{defn}

\begin{ex}[Common Examples]
Consider a set $S$.
\begin{enumerate}
\item The full quiver $K(S)$ is a loaded quiver.
\item The bouquet quiver $B(S)$ is loaded if and only if $S\neq\emptyset$.
\item The independent set of vertices $I(S)$ is loaded if and only if $S=\emptyset$.
\item The independent set of edges $M(S)$ is loaded if and only if $S=\emptyset$.
\end{enumerate}
\end{ex}

\begin{ex}
The quiver below is loaded, but not a full quiver nor a bouquet.
\[\xymatrix{
\bullet\ar@(dl,ul)^{}\ar@/^/[r]^{}\ar@/^1pc/[r]^{}	&	\bullet\ar@/^/[l]^{}\ar@(u,r)^{}\ar@(r,d)^{}\\
}\]
\end{ex}

\begin{prop}[Loaded Characterization]\label{crazycomplete}
A quiver $J$ is loaded if and only if $J$ is injective with respect to the natural inclusion of an independent set of two vertices into an independent set of one edge.
\end{prop}

\begin{proof}

($\Leftarrow$) Let $v,w\in V(J)$.  Define $\psi:\{0,1\}\to V(J)$ by $\psi(0):=v$ and $\psi(1):=w$.  Then, there is a unique quiver map $\hat{\psi}$ from $G$ to $J$ such that $V\left(\hat{\psi}\right)=\psi$.  This situation is described in the diagram below.
\[\xymatrix{
J\\
G\ar[r]_{\phi}\ar[u]^{\hat{\psi}}	&	H\\
}\]
As $J$ is injective with respect to $\phi$, there is a quiver map $\tilde{\psi}$ from $H$ to $J$ such that $\tilde{\psi}\circ\phi=\hat{\psi}$.  Let $f:=E\left(\tilde{\psi}\right)(e)$.  A calculation shows $\sigma_{J}(f)=v$ and $\tau_{J}(f)=w$.  Thus, $f\in\edges_{J}(v,w)$, meaning $G$ is loaded.

($\Rightarrow$) Consider a quiver map $\psi$ from $G$ to $J$.  Let $v:=V(\psi)(0)$, $w:=V(\psi)(1)$, and $f\in\edges_{J}(v,w)$.  Define $\hat{\psi}:\{e\}\to E(J)$ by $\hat{\psi}(e):=f$.  Then, there is a unique quiver map $\tilde{\psi}$ from $H$ to $J$ such that $E\left(\tilde{\psi}\right)=\hat{\psi}$.  A calculation shows that $V\left(\tilde{\psi}\circ\phi\right)=V(\psi)$, meaning $\tilde{\psi}\circ\phi=\psi$ by the universal property of $G$.

\end{proof}

\subsection{Mono-Injectivity}

Let $\Phi$ be the class of all monomorphisms in $\Quiv$.  From here forward, the term ``mono-injective'' will be used for being injective relative to the class of all monomorphisms.  This section will characterize the mono-injective quivers.

Since the map $\phi$ from Section \ref{motivating} is monic in $\Quiv$, every mono-injective $J$ must be loaded.  However, the inclusion of the other monic maps does not shrink the class of objects much further.

\begin{prop}[Mono-Injective Characterization]\label{mono-inj}
A quiver $J$ is mono-injective in $\Quiv$ if and only if $J$ is loaded and has at least one vertex.
\end{prop}

\begin{proof}

($\Rightarrow$) By Proposition \ref{crazycomplete}, $J$ must be loaded.  Further, let $\mathbf{0}_{V(J)}:\emptyset\to V(J)$ and $\mathbf{0}_{\{0\}}:\emptyset\to\{0\}$ be the empty functions to $V(J)$ and $\{0\}$, respectively.  Then, there is a unique quiver map $\hat{\mathbf{0}}_{V(J)}$ from $I(\emptyset)$ to $J$.  Consider the following diagram in $\Quiv$.
\[\xymatrix{
J\\
I(\emptyset)\ar[r]_{I\left(\mathbf{0}_{\{0\}}\right)}\ar[u]^{\hat{\mathbf{0}}_{V(J)}}	&	I(\{0\})\\
}\]
Since $I\left(\mathbf{0}_{\{0\}}\right)$ is monic, there is a quiver map $\psi$ from $I(\{0\})$ to $J$ such that $\psi\circ I\left(\mathbf{0}_{\{0\}}\right)=\hat{\mathbf{0}}_{V(J)}$.  Therefore, $V(\psi):\{0\}\to V(J)$, forcing $V(J)$ to be nonempty.

($\Leftarrow$) Consider the following diagram in $\Quiv$,
\[\xymatrix{
J\\
D\ar[r]_{\varphi}\ar[u]^{\psi}	&	C\\
}\]
where $\varphi$ is monic.  Then, $V(\varphi)$ and $E(\varphi)$ are one-to-one.  This will be used throughout in constructing the extension of $\psi$.

Define the following partition of the vertices and edges of $C$, where $\Ran$ denotes the range of a particular function.
\[\begin{array}{rcl}
V_{0}	&	:=	&	\Ran(V(\varphi)),\\
V_{1}	&	:=	&	V(C)\setminus V_{0},\\
E_{0}	&	:=	&	\Ran(E(\varphi)),\\
E_{1}	&	:=	&	\left\{e\in E(C):\sigma_{C}(e),\tau_{C}(e)\in V_{1}\right\},\\
E_{2}	&	:=	&	\left\{e\in E(C)\setminus E_{0}:\sigma_{C}(e),\tau_{C}(e)\in V_{0}\right\},\\
E_{3}	&	:=	&	\left\{e\in E(C):\sigma_{C}(e)\in V_{0},\tau_{C}(e)\in V_{1}\right\},\\
E_{4}	&	:=	&	\left\{e\in E(C):\sigma_{C}(e)\in V_{1},\tau_{C}(e)\in V_{0}\right\}.\\
\end{array}\]
For $V_{0}$ and $E_{0}$, $\psi$ determines their images in $J$.

Choose some $w\in V(J)$ and $f\in\edges_{J}(w,w)$ for the images of $V_{1}$ and $E_{1}$.

For each $e\in E_{2}$, then there are unique $s_{e},t_{e}\in V(D)$ such that $\sigma_{C}(e)=V(\varphi)\left(s_{e}\right)$ and $\tau_{C}(e)=V(\varphi)\left(t_{e}\right)$.  Choose $g_{e}\in\edges_{J}\left(V(\psi)\left(s_{e}\right),V(\psi)\left(t_{e}\right)\right)$ as its image.

For each $e\in E_{3}$, then there is a unique $s_{e}\in V(D)$ such that $\sigma_{C}(e)=V(\varphi)\left(s_{e}\right)$.  Choose $h_{e}\in\edges_{J}\left(V(\psi)\left(s_{e}\right),w\right)$ as its image.

For each $e\in E_{4}$, then there is a unique $t_{e}\in V(D)$ such that $\tau_{C}(e)=V(\varphi)\left(t_{e}\right)$.  Choose $i_{e}\in\edges_{J}\left(w,V(\psi)\left(t_{e}\right)\right)$ as its image.

Define $\hat{\psi}_{V}:V(C)\to V(J)$ by
\[
\hat{\psi}_{V}(v):=\left\{\begin{array}{cc}
V(\psi)(x),	&	v=V(\phi)(x), x\in V(D),\\
w,	&	v\in V_{1},\\
\end{array}\right.
\]
and $\hat{\psi}_{E}:E(C)\to E(J)$ by
\[
\hat{\psi}_{E}(e):=\left\{\begin{array}{cc}
V(\psi)(y),	&	e=E(\phi)(y), y\in E(D),\\
f,	&	e\in E_{1},\\
g_{e},	&	e\in E_{2},\\
h_{e},	&	e\in E_{3},\\
i_{e},	&	e\in E_{4}.\\
\end{array}\right.
\]
A routine check shows that $\hat{\psi}:=\left(\hat{\psi}_{V},\hat{\psi}_{E}\right)$ is a quiver map from $C$ to $J$, and $\hat{\psi}\circ\varphi=\psi$ by design.

\end{proof}

\subsection{Mono-Essential Maps and the Mono-Injective Envelope}

With mono-injective objects characterized, half of the injective envelope question is solved.  Next, mono-essential maps are characterized.  The empty quiver $I(\emptyset)$ is a singular case since every map from it is monic.  Thus, it will be considered separately.

\begin{prop}[Mono-Essential Map Characterization, $I(\emptyset)$ Case]\label{mono-ess-triv}
A quiver map $\xymatrix{I(\emptyset)\ar[r]^{\varphi} & C}$ is mono-essential if and only if $\card(V(C))\leq 1$ and $\card(E(C))\leq 1$.
\end{prop}

\begin{proof}

Given any quiver $C$, then $\hat{\mathbf{0}}_{V(C)}=\left(\mathbf{0}_{V(C)},\mathbf{0}_{E(C)}\right)$ is the unique quiver map from $I(\emptyset)$ to $C$.  Likewise, $\hat{\mathbf{1}}_{E(C)}=\left(\mathbf{1}_{V(C)},\mathbf{1}_{E(C)}\right)$ is the unique quiver map from $C$ to $B(\mathbb{1})$.  Observe that $\hat{\mathbf{0}}_{V(C)}$ is always monic, as is $\left(\mathbf{0}_{\mathbb{1}},\mathbf{0}_{\mathbb{1}}\right)=\hat{\mathbf{1}}_{E(C)}\circ\hat{\mathbf{0}}_{V(C)}$.

($\Rightarrow$) By the above fact, $\varphi=\hat{\mathbf{0}}_{V(C)}$.  Since $\varphi$ is mono-essential, $\hat{\mathbf{1}}_{E(C)}$ must be monic.  Then, $\mathbf{1}_{V(C)}$ and $\mathbf{1}_{E(C)}$ are one-to-one, giving $\card(V(C))\leq 1$ and $\card(E(C))\leq 1$.

($\Leftarrow$) Given that $\card(V(C))\leq 1$ and $\card(E(C))\leq 1$, any functions from $V(C)$ or $E(C)$ are immediately one-to-one.  Hence, every quiver map from $C$ is monic.

\end{proof}

Assuming that the vertex set is nonempty, a mono-essential map adds no vertices, and can only add an edge from $v$ to $w$ if there was not one already.

\begin{prop}[Mono-Essential Map Characterization, Nontrivial Case]\label{mono-ess}
A monic quiver map $\xymatrix{D\ar[r]^{\varphi} & C}$, where $V(D)\neq\emptyset$, is mono-essential if and only if the following conditions hold:
\begin{enumerate}
\item $V(\varphi)$ is bijective;
\item\label{mono-ess-2} if $\edges_{D}(v,w)\neq\emptyset$ for some $v,w\in V(D)$, then
\[
E(\varphi)\left(\edges_{D}(v,w)\right)=\edges_{C}\left(V(\varphi)(v),V(\varphi)(w)\right);
\]
\item\label{mono-ess-3} if $\edges_{D}(v,w)=\emptyset$ for some $v,w\in V(D)$, then
\[
\card\left(\edges_{C}\left(V(\varphi)(v),V(\varphi)(w)\right)\right)\leq 1.
\]
\end{enumerate}
\end{prop}

\begin{proof}

($\Leftarrow$) Let $\xymatrix{C\ar[r]^{\alpha} & A}\in\Quiv$ satisfy that $\alpha\circ\varphi$ is monic.  Then, $V(\alpha)\circ V(\varphi)$ and $E(\alpha)\circ E(\varphi)$ are one-to-one.  Since $V(\varphi)$ is bijective, $V(\alpha)$ is one-to-one.

Consider $e,f\in E(C)$ such that $E(\alpha)(e)=E(\alpha)(f)$.  Let $v:=\sigma_{C}(e)$ and $w:=\tau_{C}(e)$.  A calculation shows
\[\begin{array}{ccc}
V(\alpha)(v)=\left(V(\alpha)\circ\sigma_{C}\right)(f)	&	\textrm{and}	&	V(\alpha)(w)=\left(V(\alpha)\circ\tau_{C}\right)(f).\\
\end{array}\]
Since $V(\alpha)$ is one-to-one, $\sigma_{C}(f)=v$ and $\tau_{C}(f)=w$, giving $e,f\in\edges_{C}(v,w)$.  If $\edges_{D}\left(V(\varphi)^{-1}(v),V(\varphi)^{-1}(w)\right)=\emptyset$, then $\card\left(\edges_{C}(v,w)\right)=1$ by Criterion \ref{mono-ess-3}, forcing $e=f$.  Otherwise, by Criterion \ref{mono-ess-2}, there are $e_{0},f_{0}\in\edges_{D}\left(V(\varphi)^{-1}(v),V(\varphi)^{-1}(w)\right)$ such that $E(\varphi)\left(e_{0}\right)=e$ and $E(\varphi)\left(f_{0}\right)=f$.  Then,
\[
E(\alpha\circ\varphi)\left(e_{0}\right)=E(\alpha)(e)
=E(\alpha)(f)
=E(\alpha\circ\varphi)\left(f_{0}\right).
\]
Since $E(\alpha\circ\varphi)$ is one-to-one, $e_{0}=f_{0}$, yielding $e=f$.  Hence, $E(\alpha)$ is one-to-one, and $\alpha$ is monic.

($\Leftarrow$) In each case, if the condition fails, an appropriate quiver equivalence relation $\sim$ on $C$ is defined, so that the quotient map $q:C\to C/\sim$ is not monic, but $q\circ\varphi$ is.

\begin{itemize}
\item[\textbf{1 Fails:}]  Assume that there is $v\in V(C)\setminus\Ran(V(\varphi))$.  Choose $w\in\Ran(V(\varphi))$ and let $\sim_{V}$ be the equivalence relation on $V(C)$ that is merely equality on all vertices except for associating $v$ and $w$.  Letting $\sim_{E}$ be the equality relation on $E(C)$, then $\sim:=\left(\sim_{V},\sim_{E}\right)$ is easily seen to be a quiver equivalence relation on $C$.

\item[\textbf{2 Fails:}]  Assume that there are $v,w\in V(D)$, $e\in E(\varphi)(\edges_{D}(v,w))$, and $f\in E(C)$ such that $f\in\edges_{C}(V(\varphi)(v),V(\varphi)(w))\setminus\Ran(E(\varphi))$.  Let $\sim_{E}$ be the equivalence relation on $E(C)$ that is merely equality on all edges except for associating $e$ and $f$.  Letting $\sim_{V}$ be the equality relation on $V(C)$, then $\sim:=\left(\sim_{V},\sim_{E}\right)$ is easily seen to be a quiver equivalence relation on $C$.

\item[\textbf{3 Fails:}]  Assume that there are
\[\begin{array}{ccc}
v,w\in V(D)	&	\textrm{and}	&	e,f\in\edges_{C}(V(\varphi)(v),V(\varphi)(w))\\
\end{array}\]
such that $e\neq f$ and $\edges_{D}(v,w)=\emptyset$.  Let $\sim_{E}$ be the equivalence relation on $E(C)$ that is merely equality on all edges except for associating $e$ and $f$.  Letting $\sim_{V}$ be the equality relation on $V(C)$, then $\sim:=\left(\sim_{V},\sim_{E}\right)$ is easily seen to be a quiver equivalence relation on $C$.
\end{itemize}

\end{proof}

Therefore, the mono-injective envelope of a quiver $D$ must be a loaded quiver with a mono-essential quiver map from $D$.  This is accomplished by adding edges to $D$ where none already exist, making it loaded.  This process is described below as the ``loading'' of a quiver.

\begin{defn}
Given a quiver $D$, let $V_{L}:=V(D)$ and
\[
E_{L}:=\{(0,e):e\in E(D)\}\cup\{(1,v,w):v,w\in V(D),\edges_{D}(v,w)=\emptyset\}.
\]
Define $\sigma_{L},\tau_{L}:E_{L}\to V_{L}$ by
\[
\sigma_{L}(f):=\left\{\begin{array}{cc}
\sigma_{D}(e),	&	f=(0,e),\\
v,	&	f=(1,v,w),\\
\end{array}\right.
\]
and
\[
\tau_{L}(f):=\left\{\begin{array}{cc}
\tau_{D}(e),	&	f=(0,e),\\
w,	&	f=(1,v,w).\\
\end{array}\right.
\]
Then, $L(D):=\left(V_{L},E_{L},\sigma_{L},\tau_{L}\right)$ is a quiver, the \emph{loading} of $D$.
\end{defn}

\begin{ex}
Let $D$ be the quiver drawn below.
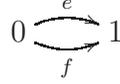
\[\xymatrix{
0\ar@/^/[r]^{e}\ar@/_/[r]_{f}	&	1\\
}\]
Then, $L(D)$ is the quiver drawn below.
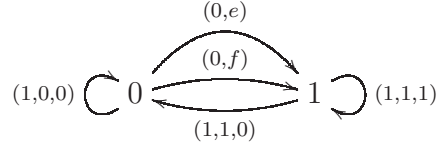
\[\xymatrix{
0\ar@(dl,ul)^{(1,0,0)}\ar@/^2pc/[rr]^{(0,e)}\ar@/^/[rr]^{(0,f)}	&	&	1\ar@(ur,dr)^{(1,1,1)}\ar@/^/[ll]^{(1,1,0)}\\
}\]
\end{ex}

For every quiver except $I(\emptyset)$, the loading characterizes the mono-injective envelope when equipped with a canonical inclusion.

\begin{thm}
Given a quiver $D$ with $V(D)\neq\emptyset$, let $L(D)$ be the loading of $D$.  Define $j_{D,V}:V(D)\to V(L(D))$ by $j_{D,V}(v):=v$ and $j_{D,E}:E(D)\to E(L(D))$ by $j_{D,E}(e):=(0,e)$.  Then, $j_{D}:=\left(j_{D,V},j_{D,E}\right)$ is a mono-essential quiver map.  Therefore, $L(D)$ equipped with $j_{D}$ is a mono-injective envelope of $D$ in $\Quiv$.
\end{thm}

\begin{proof}

By Proposition \ref{mono-inj}, $L(D)$ is mono-injective, and $j_{D}$ satisfies the conditions of Proposition \ref{mono-ess}.

\end{proof}

This theorem guarantees a mono-injective envelope for every quiver except for $I(\emptyset)$, but even $I(\emptyset)$ has a mono-injective envelope.

\begin{ex}[Mono-Injective Envelope of $I(\emptyset)$]
Consider the bouquet of one loop, $B(\mathbb{1})$.  The quiver map $\left(\mathbf{0}_{\mathbb{1}},\mathbf{0}_{\mathbb{1}}\right)$ from $I(\emptyset)$ to $B(\mathbb{1})$ is mono-essential by Proposition \ref{mono-ess-triv}.  Also, $B(\mathbb{1})$ is loaded, so this bouquet equipped with this embedding is a mono-injective envelope of $I(\emptyset)$.
\end{ex}

Thus, every quiver has a mono-injective envelope in $\Quiv$.  This fact also codifies abstractly the statement that ``every graph is a subgraph of a `complete' graph''.  Further, since a mono-injective envelope is unique up to isomorphism, any representation of it will do.  The following are a few examples of mono-injective envelopes for some common quivers.

\begin{ex}[Empty Quivers]
For any nonempty set $S$, the quiver map $\left(id_{S},\mathbf{0}_{S^{2}}\right)$ from $I(S)$ to $K(S)$ is mono-essential by Proposition \ref{mono-ess}, and $K(S)$ is loaded.  Thus, $K(S)$ with this map is a mono-injective envelope of $I(S)$.
\end{ex}

\begin{ex}[Loaded Quivers]
For any loaded quiver $D$ with $V(D)\neq\emptyset$, the identity map $id_{D}$ from $D$ to itself is mono-essential by Proposition \ref{mono-ess}.  Thus, $D$ with its identity map is a mono-injective envelope of $D$.  This includes full quivers and bouquets.
\end{ex}

\section{Projectivity}

From \cite[\S II.9.27]{joyofcats}, the dual notion of $\Phi$-injectivity is \emph{$\Phi$-projectivity}, diagrammatically described below.
\[\xymatrix{
P\ar[d]_{\psi}\ar@{..>}[dr]^{\hat{\psi}}\\
B	&	A\ar[l]^{\phi}\\
}\]
Similarly, a \emph{$\Phi$-coessential} map is the minimality condition dual to that of a $\Phi$-essential map.

This section considers projectivity of quivers with respect to the class of all epimorphisms.  From \cite[Fact 2.15]{graphtransformation}, a quiver homomorphism $\phi$ is epic if and only if both $V(\phi)$ and $E(\phi)$ are onto.  Likewise, $\phi$ is an isomorphism if and only if both $V(\phi)$ and $E(\phi)$ are bijective.

A key ingredient in this section will be the following construction.

\begin{defn2}\label{explosion}
Given a quiver $G$, a vertex $v\in V(G)$ is \emph{independent} if $\sigma_{G}^{-1}(v)=\tau_{G}^{-1}(v)=\emptyset$.  Define
\[
\indep(G):=\{v\in V(G):v\textrm{ is independent}\},
\]
the set of all independent vertices in $G$.
The \emph{explosion} of $G$ is the quiver
\[
X(G):=I(\indep(G))\coprod M(E(G)),
\]
the disjoint union of the independent set of $G$ with the edges of $G$ forced to be independent.
\end{defn2}

\begin{ex2}
Let $G$ be the quiver drawn below.
\[\xymatrix{
v\ar@(dl,ul)^{e}	&	w\ar@/^2pc/[d]^{f}\\
u	&	x\ar@/^/[u]^{g}\ar@/_/[u]_{h}\\
}\]
Then, $\indep(G)=\{u\}$.  Forcing the edge set to be independent yields $M(E(G))$ below, using the representation in Definition \ref{reflections}.
\[\xymatrix{
(0,e)\ar[d]^{e}	&	(0,f)\ar[d]^{f}	&	(0,g)\ar[d]^{g}	&	(0,h)\ar[d]^{h}\\
(1,e)	&	(1,f)	&	(1,g)	&	(1,h)\\
}\]
To draw $X(G)$, $I(\indep(G))$ will be denoted by elements of the form $(0,x)$, while those for $M(E(G))$ will have the form $(1,x)$.
\[\xymatrix{
(0,u)	&	(1,(0,e))\ar[d]^{(1,e)}	&	(1,(0,f))\ar[d]^{(1,f)}	&	(1,(0,g))\ar[d]^{(1,g)}	&	(1,(0,h))\ar[d]^{(1,h)}\\
&	(1,(1,e))	&	(1,(1,f))	&	(1,(1,g))	&	(1,(1,h))\\
}\]
\end{ex2}

There is a natural map from $X(G)$ onto the original quiver $G$.

\begin{defn2}
Given a quiver $G$, let $\kappa_{G}:\indep(G)\to V(G)$ by $\kappa_{G}(v):=v$, the usual inclusion of the independent vertices.  Likewise, let $\lambda_{G}:E(G)\to E(G)$ be the identity function on $E(G)$.  Then, there are unique quiver maps $\hat{\kappa}_{G}:I(\indep(G))\to G$ and $\hat{\lambda}_{G}:M(E(G))\to G$ such that $V\left(\hat{\kappa}_{G}\right)=\kappa_{G}$ and $E\left(\hat{\lambda}_{G}\right)=\lambda_{G}$.  Let $\iota_{1,G}$ and $\iota_{2,G}$ be the canonical inclusions of $I(\indep(G))$ and $M(E(G))$, respectively, into $X(G)$.  Then, there is a unique quiver map $p_{G}:X(G)\to G$ such that $p_{G}\circ\iota_{1,G}=\hat{\kappa}_{G}$ and $p_{G}\circ\iota_{2,G}=\hat{\lambda}_{G}$, the \emph{covering map} of $G$.
\end{defn2}

A routine check shows that $p_{G}$ is epic with $E\left(p_{G}\right)$ bijective.  The main result in this section is to show that $X(G)$ equipped with $p_{G}$ is the epi-projective cover of $G$.

\subsection{Epi-Projectivity}

From here forward, the term ``epi-projective'' will be used for being projective relative to the class of all epimorphisms.  This section will characterize the epi-projective quivers as precisely the disjoint union of an independent set of vertices with an independent set of edges.

\begin{prop}[Epi-Projective Characterization]\label{epi-proj}
A quiver $P$ is epi-projective in $\Quiv$ if and only if $P\cong I(S)\coprod M(T)$ for some sets $S$ and $T$.
\end{prop}

\begin{proof}

($\Leftarrow$) Let $S$ and $T$ be sets.  Define $P:=I(S)\coprod M(T)$, and let $\iota_{1}$ and $\iota_{2}$ be the canonical inclusions of $I(S)$ and $M(T)$, respectively, into $P$.  Consider the following diagram in $\Quiv$,
\[\xymatrix{
I(S)\ar[dr]^{\iota_{1}}	&	&	M(T)\ar[dl]_{\iota_{2}}\\
&	P\ar[d]_{\psi}\\
&	H	&	G\ar[l]^{\phi}\\
}\]
where $\phi$ is epic.  Then, both $V(\phi)$ and $E(\phi)$ are onto.  For each $s\in S$, choose $v_{s}\in V(G)$ such that $V(\phi)\left(v_{s}\right)=V\left(\psi\circ\iota_{1}\right)(s)$.  For each $t\in T$, choose $e_{t}\in E(G)$ such that $E(\phi)\left(e_{t}\right)=E\left(\psi\circ\iota_{2}\right)(t)$.  Define $\alpha:S\to V(G)$ by $\alpha(s):=v_{s}$ and $\beta:T\to E(G)$ by $\beta(t):=e_{t}$.  Then, there are unique quiver maps $\hat{\alpha}:I(S)\to G$ and $\hat{\beta}:M(T)\to G$ such that $V\left(\hat{\alpha}\right)=\alpha$ and $E\left(\hat{\beta}\right)=\beta$.  Furthermore, there is a unique quiver map $\gamma:P\to G$ such that $\gamma\circ\iota_{1}=\hat{\alpha}$ and $\gamma\circ\iota_{2}=\hat{\beta}$.

A calculation shows that for $s\in S$ and $t\in T$,
\[\begin{array}{ccc}
V\left(\phi\circ\gamma\circ\iota_{1}\right)(s)=V\left(\psi\circ\iota_{1}\right)(s)
&
\textrm{and}
&
E\left(\phi\circ\gamma\circ\iota_{2}\right)(t)
=E\left(\psi\circ\iota_{1}\right)(t).\\
\end{array}\]
By the universal properties of $I$ and $M$, $\phi\circ\gamma\circ\iota_{1}=\psi\circ\iota_{1}$ and $\phi\circ\gamma\circ\iota_{2}=\psi\circ\iota_{2}$.  By the universal property of the disjoint union, $\phi\circ\gamma=\psi$.

($\Rightarrow$) This direction of the proof will show that the covering map $p_{P}:X(P)\to P$ is an isomorphism.  This will be done by creating an inverse mapping.  Consider the following diagram in $\Quiv$.
\[\xymatrix{
&	X(P)\ar[d]^{p_{P}}\\
P\ar[r]_{id_{P}}	&	P\\
}\]
Since $P$ is epi-projective, there is a quiver map $\psi:P\to X(P)$ such that $p_{P}\circ\psi=id_{P}$.  As a result, $V\left(p_{P}\right)\circ V(\psi)=id_{V(P)}$ and $E\left(p_{P}\right)\circ E(\psi)=id_{E(P)}$.  This guarantees that $V(\psi)$ is one-to-one.  Since $E\left(p_{P}\right)$ is bijective, $E(\psi)=E\left(p_{P}\right)^{-1}$ is too.  A calculation shows that for all $e\in E(P)$,
\[\begin{array}{ccc}
V(\psi)\left(\sigma_{P}(e)\right)
=(1,(0,e))
&
\textrm{and}
&
V(\psi)\left(\tau_{P}(e)\right)
=(1,(1,e)).\\
\end{array}\]
For $v\in\indep(P)$,
\[
v=id_{V(P)}(v)
=\left(V\left(p_{P}\right)\circ V(\psi)\right)(v).
\]
Thus, $V(\psi)(v)\in V\left(p_{P}\right)^{-1}(v)=\{(0,v)\}$, giving $V(\psi)(v)=(0,v)$.  Therefore, $V(\psi)$ is onto, so $\psi$ is an isomorphism between $P$ and $X(P)$.

\end{proof}

With this characterization, $X(G)$ is guaranteed to be epi-projective for every quiver $G$.

\subsection{Epi-Coessential Maps and the Epi-Projective Cover}

Next, epi-coessential maps are characterized, which will consequentially yield that $p_{G}$ is epi-coessential.  Specifically, an epi-coessential map must be bijective on edges and independent vertices.

\begin{prop}[Epi-Coessential Map Characterization]\label{epi-coess}
An epic quiver map $\xymatrix{G\ar[r]^{\phi} & H}$ is epi-coessential if and only if the following conditions hold:
\begin{enumerate}
\item $E(\phi)$ is bijective;
\item if $v\in\indep(G)$, then $V(\phi)(v)\in\indep(H)$;
\item\label{epi-coess-3} if $w\in\indep(H)$, there is a unique $v\in\indep(G)$ such that $V(\phi)(v)=w$.
\end{enumerate}
\end{prop}

\begin{proof}

($\Leftarrow$) Let $\xymatrix{A\ar[r]^{\alpha} & G}\in\Quiv$ satisfy that $\phi\circ\alpha$ is epic.  Then, $V(\phi)\circ V(\alpha)$ and $E(\phi)\circ E(\alpha)$ are onto.  Since $E(\phi)$ is bijective, $E(\alpha)$ is onto.

Consider $v\in V(G)$.  If there is $e\in E(G)$ such that $\sigma_{G}(e)=v$, then there is $f\in E(A)$ such that $E(\alpha)(f)=e$.  Note that
\[
V(\alpha)\left(\sigma_{A}(f)\right)=\left(\sigma_{G}\circ E(\alpha)\right)(f)
=\sigma_{G}(e)
=v.
\]
A similar situation occurs if $v=\tau_{G}(e)$ for some $e\in E(G)$.

If $v\in\indep(G)$, then $V(\phi)(v)\in\indep(H)$.  Then, there is $u\in V(A)$ such that $V(\phi\circ\alpha)(u)=V(\phi)(v)$.  If there was $e\in E(G)$ such that $\sigma_{G}(e)=\alpha(u)$, then a calculation shows $\left(\sigma_{H}\circ E(\phi)\right)(e)=V(\phi)(v)$, contradicting that $V(\phi)(v)\in\indep(H)$.  If $\tau_{G}(e)=\alpha(u)$ for some $e\in E(G)$, a similar contradiction results.  Therefore, $V(\alpha)(u)\in\indep(G)$, and $V(\alpha)(u)=v$ by Criterion \ref{epi-coess-3}.

Thus, $V(\alpha)$ is onto, and $\alpha$ is epic.

($\Leftarrow$) In each case, if the condition fails, an appropriate subquiver $N$ within $G$ is defined, so that the inclusion map $\iota:N\to G$ is not epic, but $\phi\circ\iota$ is.

\begin{itemize}
\item[\textbf{1 Fails:}] Assume there are $e,f\in E(G)$ such that $e\neq f$ and $E(\phi)(e)=E(\phi)(f)$.  Let $V_{N}:=V(G)$ and $E_{N}:=E(G)\setminus\{f\}$.  One can check that the restrictions of $\sigma_{G}$ and $\tau_{G}$ to $E_{N}$ map into $V_{N}$.

\item[\textbf{2 Fails:}] Assume $E(\phi)$ is bijective and there is $v\in\indep(G)$ and $e\in E(H)$ such that $\sigma_{H}(e)=V(\phi)(v)$.  Let $V_{N}:=V(G)\setminus\{v\}$ and $E_{N}:=E(G)$.  One can check that the restrictions of $\sigma_{G}$ and $\tau_{G}$ to $E_{N}$ map into $V_{N}$.

The case when there is $e\in E(H)$ such that $\tau_{H}(e)=E(\phi)(v)$ can be treated similarly.

\item[\textbf{3 Fails:}]Assume there are $v,w\in\indep(G)$ such that $v\neq w$ and $V(\phi)(v)=V(\phi)(w)$.  Let $V_{N}:=V(G)\setminus\{w\}$ and $E_{N}:=E(G)$.  One can check that the restrictions of $\sigma_{G}$ and $\tau_{G}$ to $E_{N}$ map into $V_{N}$.
\end{itemize}

\end{proof}

With this characterization, $p_{G}$ is epi-coessential by design.  This yields the characterization of the epi-projective cover of any quiver $G$.

\begin{cor}
Given any quiver $G$, $X(G)$ equipped with $p_{G}$ is the epi-projective cover of $G$.
\end{cor}

This fact codifies abstractly the statement that ``every graph is a quotient of an `independent' graph''.  Furthermore, since an epi-projective cover is unique up to isomorphism, any representation of it will do.  The following are a few examples of epi-projective covers for some common quivers.

\begin{ex}[Independent Sets]
For any sets $S$ and $T$, let
\[
P:=I(S)\coprod M(T).
\]
The quiver map $id_{P}$ from $P$ to itself is epi-coessential by Proposition \ref{epi-coess}.  Thus, $P$ with this map is an epi-projective cover of $P$.
\end{ex}

\begin{ex}[Bouquets]
Given any nonempty set $S$, let
\[
\mathbf{1}_{V(M(S))}:V(M(S))\to\mathbb{1}
\]
be the constant map.  Then, the quiver map $\left(\mathbf{1}_{V(M(S))},id_{S}\right)$ from $M(S)$ to $B(S)$ is epi-coessential by Proposition \ref{epi-coess}.  Thus, $M(S)$ with this map is an epi-projective cover of $B(S)$.
\end{ex}

\begin{ex}[Full Quivers]
Given any set $S$, define $\rho:V\left(M\left(S^{2}\right)\right)\to S$ by
\[
\rho(v):=\left\{\begin{array}{cc}
s,	&	v=(0,(s,t)),\\
t,	&	v=(1,(s,t)).\\
\end{array}\right.
\]
Then, the quiver map $\left(\rho,id_{S^{2}}\right)$ is epi-coessential from $M\left(S^{2}\right)$ to $K(S)$.  Thus, $M\left(S^{2}\right)$ with this map is an epi-projective cover of $K(S)$.
\end{ex}

\end{document}